\documentclass[12pt,reqno]{amsart}
 
     \usepackage[all]{xy}
     \usepackage{graphicx}
	
     	\usepackage{amssymb}
     \usepackage{amsmath}
    	\usepackage{amsfonts}
	\usepackage{latexsym}   		
	
    \newtheorem{thm}{Theorem}[section]

    \newtheorem{lem}[thm]   {Lemma}
    \newtheorem{cor}[thm]   {Corollary}

    \newtheorem{defn}[thm]  {Definition}

    \newtheorem{prop}[thm]  {Proposition}

\theoremstyle{remark}
    
    \newtheorem{rem}[thm]   {Remark}

\newcommand{\secref}[1]{Section~\ref{#1}}
\newcommand{\thmref}[1]{Theorem~\ref{#1}}
\newcommand{\propref}[1]{Proposition~\ref{#1}}
\newcommand{\lemref}[1]{Lemma~\ref{#1}}
\newcommand{\corref}[1]{Corollary~\ref{#1}}

    \newcommand{\F}        {\mathcal{F}}

    \newcommand{\ZZ}    {\mathbb{Z}}

    \newcommand{\RR}    {\mathbb{R}}
    
    \newcommand{\CC}    {\mathbb{C}}

    \newcommand{\QQ}    {\mathbb{Q}}

\begin{document}
	


\title{On the Structure of Co-K\"ahler Manifolds}

\author[G. Bazzoni]{Giovanni Bazzoni}
\address{Instituto de Ciencias Matem\'aticas CSIC-UAM-UC3M-UCM \\
Calle Nicol\'as Cabrera 13-15, Campus Cantoblanco\\
28049 Madrid, Spain}
\email{gbazzoni@icmat.es}

\author[J. Oprea]{John Oprea}
\address{Department of Mathematics\\
Cleveland State University\\
Cleveland OH \\
44115  USA}
\email{j.oprea@csuohio.edu}

\begin{abstract}
By the work of Li, a compact co-K\"ahler manifold $M$ is a mapping torus
$K_\varphi$, where $K$ is a K\"ahler manifold and $\varphi$ is a
Hermitian isometry. We show here that there is always
a finite cyclic cover $\overline M$ of the form
$\overline M \cong K \times S^1$, where $\cong$ is equivariant
diffeomorphism with respect to an action of $S^1$ on $M$ and
the action of $S^1$ on $K \times S^1$ by translation
on the second factor. Furthermore, the covering transformations
act diagonally on $S^1$, $K$ and are translations on the $S^1$ factor.
In this way, we see that, up to a finite cover, all compact co-K\"ahler manifolds
arise as the product of a K\"ahler manifold and a circle.
\end{abstract}

\maketitle
\pagestyle{myheadings}

\markboth{{G. Bazzoni and J. Oprea}}
{{On the Structure of co-K\"ahler Manifolds}}


\vskip1pt{\small\textbf{MSC classification [2010]}: Primary 53C25;
Secondary 53B35, 53C55, 53D05.}\vskip0.5pt
\vskip1pt{\small\textbf{Key words and phrases}: co-K\"ahler manifolds, mapping tori.}\vskip10pt

\section{Recollections on Co-K\"ahler Manifolds}\label{sec:reco}
In \cite{Li}, H. Li recently gave a structure result for compact co-K\"ahler
manifolds stating that such a manifold is always a K\"ahler mapping torus
(see \secref{sec:maptor}). In this paper, using Li's characterization, we give 
another type of structure theorem for co-K\"ahler manifolds based on classical
results in \cite{CR,Op,Sad,Wel}. As such, much of this paper is devoted
to showing how the interplay between the known geometry and the known
topology of co-K\"ahler manifolds creates beautiful structure. Basic results on
co-K\"ahler manifolds themselves come from \cite{CDM} (see also \cite{FV})
\footnote{The authors of \cite{CDM}
use the term \emph{cosymplectic} for Li's co-K\"ahler because they view these
manifolds as odd-dimensional versions of symplectic manifolds --- even as far
as being a convenient setting for time-dependent mechanics \cite{DT}. Li's
characterization, however, makes clear the true underlying K\"ahler structure,
so we have chosen to follow his terminology.}.\\

Let $(M^{2n+1},J,\xi,\eta,g)$ be an \emph{almost contact metric manifold}
given by the conditions
\begin{equation}\label{eq:1}
J^2 = -I + \eta \otimes \xi,\quad \eta(\xi)=1, \quad g(JX,JY)=g(X,Y)-\eta(X)\eta(Y),
\end{equation}
for vector fields $X$ and $Y$, $I$ the identity transformation on $TM$ and $g$ a
Riemannian metric. Here, $\xi$ is a vector field as well, $\eta$
is a $1$-form and $J$ is a tensor of type $(1,1)$. A local $J$-basis 
$\{X_1,\ldots,X_n,JX_1,\ldots,JX_n,\xi\}$
may be found with $\eta(X_i)=0$ for $i=1,\ldots,n$. The \emph{fundamental
$2$-form} on $M$ is given by
$$\omega(X,Y) = g(JX,Y),$$
and if $\{\alpha_1,\ldots,\alpha_n,\beta_1,\ldots,\beta_n,\eta\}$ is
a local $1$-form basis dual to the local $J$-basis, then
$$\omega = \sum_{i=1}^n \alpha_i \wedge \beta_i.$$
Note that $\imath_\xi \omega = 0$.

\begin{defn}\label{def:cosymp}
The geometric structure $(M^{2n+1},J,\xi,\eta,g)$ is a {\bf co-K\"ahler} structure on $M$ if
$$[J,J]+2\, d\eta \otimes \xi=0 \ \, {\rm and}\ \, d\omega=0=d\eta$$
or, equivalently, $J$ is parallel with respect to the metric $g$.
\end{defn}

A crucial fact that we use in our result is that, on a co-K\"ahler manifold, the vector 
field $\xi$ is Killing and parallel and the 1-form $\eta$ is harmonic. This fact is well 
known, but we were not able to find a direct proof in the literature, so we give one here.

\begin{lem}\label{lem:parallel}
On a co-K\"ahler manifold, the vector field $\xi$ is Killing and parallel.
Furthermore, the 1-form $\eta$ is a harmonic form.
\end{lem}

\begin{proof}
The normality condition implies that $L_\xi J=0$ (see \cite{Bl}); in parti\-cular, 
$[\xi,JX]=J[\xi,X]$ for every vector field $X$ on $M$. Compatibility of the metric 
$g$ with $J$ is expressed by the right-hand relation in (\ref{eq:1}); with 
$\omega(X,Y)=g(JX,Y)$, it yields
\begin{equation}\label{eq:101}
g(X,Y)=\omega(X,JY)+\eta(X)\eta(Y).
\end{equation}
By definition,
\begin{equation}\label{eq:102}
(L_\xi g)(X,Y)=\xi g(X,Y)-g([\xi,X],Y)-g(X,[\xi,Y]).
\end{equation}
Substituting (\ref{eq:101}) in (\ref{eq:102}), we obtain
\begin{align*}
(L_\xi g)(X,Y)= & \ \xi \omega(X,Y)+\xi(\eta(X)\eta(Y))-\omega([\xi,X],JY)-\eta([\xi,X])\eta(Y)+\\
& - \ \omega(X,J[\xi,Y])-\eta(X)\eta([\xi,Y])=\\
= & \ \xi \omega(X,Y)-\omega([\xi,X],JY)-\omega(X,[\xi,JY])+(\xi\eta(X))\eta(Y)+\\
& + \eta(X)(\xi\eta(Y))-\eta([\xi,X])\eta(Y)-\eta(X)\eta([\xi,Y])=\\
= & \ (L_\xi\omega)(X,JY)+\eta(X)(\xi\eta(Y)-\eta([\xi,Y]))+\\
 & + \ \eta(Y)(\xi\eta(X)-\eta([\xi,X]))=\\
= & \ \eta(X)(d\eta(\xi,Y)+Y\eta(\xi))+\eta(Y)(d\eta(\xi,X)+X\eta(\xi))=\\
= & \ 0.
\end{align*}
The last equalities follow from these facts:
\begin{itemize}
\item since $\omega$ is closed and $\imath_\xi\omega=0$, $L_\xi\omega=0$ by Cartan's 
magic formula;
\item $d\eta=0$;
\item as $\eta(\xi)\equiv 1$, one has $X\eta(\xi)=Y\eta(\xi)=0$.
\end{itemize}
This proves that $\xi$ is a Killing vector field. In order to show that $\xi$ is parallel, 
we use the following formula for the covariant derivative $\nabla$ of the Levi-Civita 
connection of $g$; for vector fields $X,Y,Z$ on $M$, one has
\begin{align}\label{eq:103}
2g(\nabla_X Y,Z)=& \ Xg(Y,Z)+Yg(X,Z)-Zg(X,Y)+\\  \nonumber
&+ g([X,Y],Z)+g([Z,X],Y)-g([Y,Z],X).
\end{align}
Setting $Y=\xi$ in (\ref{eq:103}) and recalling that, on any almost contact metric manifold,
$g(X,\xi)=\eta(X)$, we obtain
\begin{align*}
2g(\nabla_X\xi,Z)= & \ Xg(\xi,Z)+\xi g(X,Z)-Zg(X,\xi)+g([X,\xi],Z)+\\
& + \ g([Z,X],\xi)-g([\xi,Z],X)=\\
=& \ \xi g(X,Z)-g([\xi,X],Z)-g([\xi,Z],X)+X\eta(Z)+\\
& -\ Z\eta(X)-\eta([X,Z])=\\
= & \ (L_\xi g)(X,Z)+d\eta(X,Z)=\\
= & \ 0.
\end{align*}
Since $X$ and $Z$ are arbitrary it follows that $\nabla\xi=0$.

To prove that $\eta$ is harmonic, we rely on the following result: a vector field 
on a Riemannian manifold $(M,g)$ is Killing if and only if the dual 1-form is co-closed. 
For a proof, see for instance \cite[page 107]{Go}. Applying this to $\xi$, we see that 
$\eta$ co-closed; since it is closed, it is harmonic.
\end{proof}

\noindent \lemref{lem:parallel} will be a key point in our structure theorem
below. In fact, in \cite{Li}, it is shown that we can replace $\eta$ by a
harmonic integral form $\eta_\theta$ with dual parallel vector field
$\xi_\theta$ and associated metric $g_\theta$, $(1,1)$-tensor
$J_\theta$ and closed $2$-form $\omega_\theta$ with $i_{\xi_\theta}\omega_\theta
=0$. Then we have the following (see \secref{sec:maptor} for definitions).

\begin{thm}[\cite{Li}] \label{thm:maptor}
With the structure $(M^{2n+1},J_\theta,\xi_\theta,\eta_\theta,g_\theta)$,
there is a compact K\"ahler manifold $(K,h)$ and a Hermitian isometry
$\psi\colon K \to K$ such that $M$ is diffeomorphic to the mapping torus
$$K_\psi = \frac{K \times [0,1]}{(x,0)\sim (\psi(x),1)}$$
with associated fibre bundle $K \to M=K_\psi \to S^1$.
\end{thm}

An important ingredient in Li's theorem is a result of Tischler (see \cite{Ti}) 
stating that a compact manifold admitting a non-vanishing closed 1-form fibres 
over the circle. The above result indicates that co-K\"ahler manifolds are very special 
types of manifolds. However it can be very difficult to see whether a manifold is
such a mapping torus. In this paper, we will give another characterization of
co-K\"ahler manifolds which we hope will allow an easier identification.

\section{Parallel Vector Fields}\label{sec:parallel}
From now on, when we write a co-K\"ahler structure $(M^{2n+1},J,\xi,\eta,g)$,
we shall mean Li's associated integral and parallel structures. Let's now
employ an argument that goes back to \cite{Wel}, but which was resurrected
in \cite{Sad}. Consider the parallel vector field $\xi$ and its associated
flow $\phi_t$. Because $\xi$ is Killing, each $\phi_t$ is an isometry of
$(M,g)$. Therefore, in the isometry group ${\rm Isom}(M,g)$, the subgroup
generated by $\xi$, $C$, is singly generated. Since $M$ is compact, so is
${\rm Isom}(M,g)$ and this means that $C$ is a torus. Using harmonic forms
and the Albanese torus, Welsh \cite{Wel} actually shows that there is a
subtorus $T \subseteq C$ such that $M = T \times_G F$ where $G \subset T$
is finite and $F$ is a manifold. Following Sadowski \cite{Sad}, we can
modify the argument as follows.

Let $S^1 \subseteq C \subset {\rm Isom}(M,g)$ have associated vector field
$Y$. Because $S^1$ acts on $(M,g)$ by isometries, the vector field $Y$
is Killing. Now, we can choose $Y$ as close to $\xi$ as we like, so at
some point $x_0 \in M$, since $\eta(\xi)(x_0) \not = 0$, then
$\eta(Y)(x_0)\not = 0$ as well. But $\eta$ is harmonic and $Y$ is Killing,
so this means that $\eta(Y)(x)\not = 0$ for all $x \in M$. Hence, we may
take $\eta(Y)(x) > 0$ for all $x \in M$. Now let $\sigma$
be an orbit of the $S^1$ action. Then
$$\int_\sigma \eta = \int_0^1 \eta\left(\frac{d\sigma}{dt}\right)\, dt
= \int\eta(Y) \, dt > 0.$$
This says that the orbit map $\alpha\colon S^1 \to M$ defined by
$g \mapsto g \cdot x_0$ induces a non-trivial composition of homomorphisms
$$H_1(S^1;\mathbb R)  \stackrel{\alpha_*}{\to} H_1(M;\mathbb R)
\stackrel{\eta}{\to} H_1(S^1;\mathbb R),$$
where $d\eta=0$ defines an integral cohomology class $\eta \in H^1(M;\mathbb Z)
\cong [M,S^1]$. Here we use the standard identification of degree $1$ cohomology
with homotopy classes of maps from $M$ to $S^1$. Since $H_1(S^1;\mathbb Z)=\mathbb Z$,
this means that the integral homomorphism $\alpha_* \colon H_1(S^1;\mathbb Z)
\to H_1(M;\mathbb Z)$ is \emph{injective}. Such an action is said to be
\emph{homologically injective} (see \cite{CR}). Hence, we have

\begin{prop}\label{prop:hominjtoract}
A co-K\"ahler manifold $(M^{2n+1},J,\xi,\eta,g)$ with integral structure
supports a smooth homologically injective $S^1$ action.
\end{prop}

In fact, it can be shown that there is a homologically injective $T^k$ action
on $M$, where $T^k$ is Welsh's torus $T$. However, we shall focus on the
$S^1$-case since this will allow a connection to Li's mapping torus result.

\section{Sadowski's Transversally Equivariant Fibrations}\label{sec:crsplit}
Homologically injective actions were first considered by P. Conner and F. Raymond
in \cite{CR} (also see \cite{LR}) and were shown to lead to topological
product splittings up to finite cover (also see \cite{Op}). Homological
injectivity for a circle action is very unusual and this points out the extremely
special nature of co-K\"ahler manifolds. Here we want to make
use of the results in \cite{Sad} to achieve smooth splittings for co-K\"ahler
manifolds up to a finite cover. We will state the results of \cite{Sad}
only for the case we are interested in: namely, a mapping torus bundle
$M \to S^1$.

Let's begin by recalling that a bundle map $p\colon M \to S^1$ is a
\emph{transversally equivariant fibration} if there is a smooth $S^1$-action
on $M$ such that the orbits of the action are transversal to the fibres of
$p$ and $p(t\cdot x)-p(x)$ depends on $t \in S^1$ only. This latter condition is
simply the usual equivariance condition if we take an appropriate action of
$S^1$ on itself (see \cite[Remark 1.1]{Sad}). Sadowski's key lemma is the following.

\begin{lem}[{\cite[Lemma 1.3]{Sad}}]\label{lem:sadowski}
Let $p\colon M \to S^1$ be a smooth $S^1$-equivariant bundle map. Then the following
are equivalent;
\begin{enumerate}
\item The orbits of the $S^1$-action are transversal to the fibres of $p$:
\item $p_* \circ \alpha_*\colon \pi_1(S^1) \to \pi_1(S^1)$ is injective, where
$\alpha\colon S^1 \to M$ is the orbit map;
\item One orbit of the $S^1$-action is transversal to a fibre of $p$ at a
point $x_0 \in M$.
\end{enumerate}
\end{lem}

\begin{rem}\label{rem:sadequiv} Note the following. \newline
\vspace{-14pt}
\begin{enumerate}
\item Lemmas 1.1 and 1.2 of \cite{Sad} show that, in the situation of
\propref{prop:hominjtoract}, $\eta\colon M \to S^1$ is a transversally equivariant
bundle map.
\vskip3pt
\item Note also that, because $\pi_1(S^1) \cong H_1(S^1;\ZZ) \cong \ZZ$, the second
condition of \lemref{lem:sadowski} is really saying that the action is homologically
injective.
\end{enumerate}
\end{rem}

As pointed out in \cite{Li}, \emph{every smooth fibration $K\to M\stackrel{p}{\to} S^1$
can be seen as a mapping torus of a certain diffeomorphism} $\varphi\colon K\to K$,
(also see \propref{mapping:torus} below). The following is a distillation of
\cite[Proposition 2.1 and Corollary 2.1]{Sad}
in the case of a circle action.

\begin{thm}\label{thm:coverthm}
Let $M \stackrel{p}{\to} S^1$ be a smooth bundle projection from a smooth
closed manifold $M$ to the circle.
The following are equivalent:
\begin{enumerate}
\item The structure group of $p$ can be reduced to a finite cyclic group
$G=\ZZ_m \subseteq \pi_1(S^1)/({\rm Im}(p_*\circ \alpha_*))$
(i.e. the diffeomorphism $\varphi$ associated to the mapping torus $M \stackrel{p}{\to}
S^1$ has finite order);
\vskip3pt
\item The bundle map $p$ is transversally equivariant with respect to an
$S^1$-action on $M$, $A\colon S^1 \times M \to M$.
\end{enumerate}
Moreover, assuming (1) and (2), there is a finite $G$-cover $K \times S^1 \to M$
given by the action $(k,t)\mapsto A_t(k)$, where $G$ acts diagonally and by
translations on $S^1$.
\end{thm}

\begin{proof}[Sketch of Proof {\rm (\cite{Sad})}]

($ 1 \Rightarrow 2$)  The bundle is classified by a map $S^1 \to
BG$ or, equivalently, by an element of $\pi_1(BG)=G=\ZZ_m$ (since $G$ is
abelian). Now $M$ may be written as a
mapping torus $K_\varphi$ for some diffeomorphism
$\varphi \in {\rm Diffeo}(K)$ of order $m$. (So $G$ is
the structure group of a mapping torus bundle). Define an $S^1$-action
by $A\colon S^1 \times M \to M$, $A(t,[k,s])=[k,s+mt]$. (Geometrically,
the action is simply winding around the mapping torus $m$ times until we are
back to the identity $\varphi^m$). Clearly, the action is transversally equivariant.

($ 2 \Rightarrow 1$)  Let $A_t\colon M \to M$ be the $S^1$-action
such that $p$ is transversally equivariant. Let $K$ be the fibre of $p$ and let
$$G=\{g \in S^1\,|\, A_g(K)=K\}.$$
Now, because orbits of the action are transversal to the fibre, $G$ is a
proper closed subgroup of $S^1$. Hence, $G=\ZZ_m=\langle g\,|\,g^m=1\rangle$
for some positive integer $m$. Also note that the transversally equivariant
condition saying $p(A_t(x))-p(x)$ only depends on $t$ implies that the
action carries fibres of $p$ to fibres of $p$. Moreover, fibres are
then mapped back to themselves by $G$. Hence, letting $G$ act diagonally on
$K \times S^1$ and by translations on $S^1$, we see that the action is free
and its restriction $A|\colon K \times S^1 \to M$ is a finite $G$-cover.
Now, if we take the piece of the orbit from
$x_0 \in K$ to $A_g(x_0)$ for fixed $x_0$ and $g \in G$, the projection
to $S^1$ gives an element in $\pi_1(S^1)=\ZZ$. Because the full orbit is
strictly longer than this piece, we see that the corresponding element
in $\pi_1(S^1)=\ZZ$ can only be in ${\rm Im}(p_*\circ \alpha_*)$ if $g = 1$.
Hence, $G \subseteq \pi_1(S^1)/({\rm Im}(p_*\circ \alpha_*))$ which is
finite due to homological injectivity.
\end{proof}

We then have the following consequence for co-K\"ahler manifolds from
\propref{prop:hominjtoract} and \thmref{thm:coverthm}.

\begin{thm}\label{thm:cosympsplit}
A compact co-K\"ahler manifold $(M^{2n+1},J,\xi,\eta,g)$ with integral structure
and mapping torus bundle $K \to M \to S^1$
splits as $M \cong S^1 \times_{\mathbb Z_m} K$,
where $S^1 \times K \to M$ is a finite cover with structure group
$\mathbb Z_m$ acting diagonally and by translations on the
first factor. Moreover, $M$ fibres over the circle $S^1/(\mathbb Z_m)$
with finite structure group.
\end{thm}

Note that \thmref{thm:cosympsplit} provides the following.

\begin{cor}\label{cor:diagram cover}
For a compact co-K\"ahler manifold $(M^{2n+1},J,\xi,\eta,g)$ with integral structure
and mapping torus bundle $K \to M \to S^1$, there is a commutative diagram of
fibre bundles:
$$\xymatrix{
K \ar[d]_= \ar[r] & S^1 \times K \ar[d]_{\times m} \ar[r] & S^1 \ar[d]_{\times m}\\
K \ar[r] & K_\psi \ar[r] & S^1  \quad .    }
$$
where $K_\psi \cong M$ according to \thmref{thm:maptor} and the notation
${\times m}$ denotes an $\ZZ_m$-covering.
\end{cor}

\begin{rem}
Although we have used the very special results of \cite{Sad} above, observe that
a version of \thmref{thm:cosympsplit} may be proved in the continuous case
using the Conner-Raymond Splitting Theorem \cite{CR}. In this case, we obtain a
finite cover $S^1 \times Y \to M$, where $Y \to K$ is a homotopy equivalence.
This type of result affords a possibility of weakening the stringent assumptions
on co-K\"ahler manifolds with a view towards homotopy theory rather than geometry.
\end{rem}

\section{Betti Numbers}\label{sec:betti}
A main result of \cite{CDM} was the fact that the Betti numbers of co-K\"ahler manifolds
increase up to the middle dimension: $b_1 \leq b_2 \leq \ldots\leq b_n=b_{n+1}$ for $M^{2n+1}$.
The argument in \cite{CDM} was difficult, involving Hodge theory and a type of Hard
Lefschetz Theorem for co-K\"ahler manifolds. In \cite{Li}, the mapping torus description of
co-K\"ahler manifolds yielded the result topologically through homology properties of
the mapping torus. Here, we would like to see the Betti number result as a natural
consequence of \thmref{thm:cosympsplit}. Recall a basic result
from covering space theory.

\begin{lem}\label{lem:covercoho}
If $\overline X \to X$ is a finite $G$-cover, then
$$H^*(X;\QQ) = H^*(\overline X;\QQ)^G,$$
where the designation $H^G$ denotes the fixed algebra under the action of the
covering transformations $G$.
\end{lem}

In order to see the Betti number relations, we need to know that
the ``K\"ahler class'' on $K$ is invariant under the covering
transformations. The following result guarantees that such a class
exists.

\begin{lem}\label{lem:invclass}
There exists a class $\bar{\omega} \in H^2(K;\RR)^G\subset H^2(S^1\times K;\RR)$ which pulls back
to $\omega \in H^2(K;\RR)$ via the inclusion $K \to S^1\times K$
contained in  \corref{cor:diagram cover}.
\end{lem}

\begin{proof}
Let $\theta\colon S^1 \times K \to M$ denote the $G=\ZZ_m$-
cover of \thmref{thm:cosympsplit} and \corref{cor:diagram cover}.
Then $\theta^* \omega = \eta \times \alpha + \bar{\omega}$, where
$\eta$ generates $H^1(S^1;\RR)$, $\alpha \in H^1(K;\RR)$ and
$\bar{\omega} \in H^2(K;\RR)$. Note that $\bar{\omega}$
pulls back to $\omega \in H^2(K;\RR)$. Also, $\theta^*\omega$ is
$G$-invariant, so for each $g \in G$, we have
\begin{align*}
\alpha \times \eta + \bar{\omega} & = g^*(\alpha \times \eta +
\bar{\omega}) \\
& = g^*(\alpha) \times g^*(\eta) + g^*(\bar{\omega}) \\
&= g^*(\alpha) \times \eta + g^*(\bar{\omega}),
\end{align*}
using the fact that $G$ acts on $K \times S^1$ diagonally and
homotopically trivially on $S^1$. We then get
$$
(\alpha-g^*(\alpha)) \times \eta=g^*(\bar{\omega})-\bar{\omega}.
$$
This also means that
$g^*(\bar{\omega})-\bar{\omega} \in H^2(K;\RR)$
and $(\alpha-g^*(\alpha)) \times \eta \in H^1(K;\RR) \otimes
H^1(S^1)$. Thus, the only way the equality above can hold is that
both sides are zero. Hence, $\bar{\omega}$ is $G$-invariant.
\end{proof}

\begin{thm}\label{thm:betti}
If $(M^{2n+1},J,\xi,\eta,g)$ is a compact co-K\"ahler manifold with integral structure
and splitting $M \cong K \times_{\mathbb Z_m} S^1$, then
$$H^*(M;\RR) = H^*(K;\RR)^G \otimes H^*(S^1;\RR),$$
where $G = \ZZ_m$. Hence, the Betti numbers of $M$ satisfy:
\vskip5pt

\noindent (1)\quad $b_s(M) = \overline b_s(K) + \overline b_{s-1}(K)$, where
$\overline b_s(K)$ denotes the dimension of $G$-invariant cohomology $H^s(K;\RR)^G$;
\vskip7pt
\noindent (2)\quad $b_1(M) \leq b_2(M) \leq \ldots \leq b_n(M) = b_{n+1}(M)$.
\end{thm}

\begin{proof}
\lemref{lem:covercoho} and the fact that $G$ acts by translations
(so homotopically trivially) on $S^1$ produce $H^*(M;\RR) = H^*(K;\RR)^G \otimes
H^*(S^1;\RR)$. If we denote the Betti numbers of the $G$-invariant cohomology
by $\overline b$, then the tensor product splitting gives
$$b_s(M) = \overline b_s(K) + \overline b_{s-1}(K),$$
using the fact that $\widetilde H^1(S^1;\RR)=\RR$ and vanishes otherwise.

Let $\{\alpha_1, \ldots,\alpha_k\}$ be a basis for $H^{s-2}(K;\RR)^G$.
According to \lemref{lem:invclass}, the class $\omega \in H^2(M;\RR)$, which comes from $H^2(K;\RR)$,
provides a $G$-invariant class in $H^2(K;\RR)$. Furthermore,
since $K$ is compact K\"ahler, $H^*(K;\RR)$ obeys the Hard Lefschetz Property with respect
to $\omega$. Namely, for $j \leq n$, multiplication by powers of $\omega$,
$$
\cdot \ \omega^{n-j} \colon H^j(K;\RR) \to H^{2n-j}(K;\RR),
$$
is an isomorphism. In particular, this means that multiplication by each
power $\omega^{s}$, $s \leq n-j$, must be injective. Therefore, for any
$s \leq n$, we have an injective homomorphism $\cdot \ \omega\colon
H^{s-2}(K;\RR) \to H^s(K;\RR)$. Thus, since $\omega \in H^2(K;\RR)^G$,
we obtain a linearly independent set $\{\omega \alpha_1,\ldots,\omega\alpha_k\}
\subset H^s(K;\RR)^G$. But then we see that, for all $s \leq n$,
$$\overline b_{s-2}(K) \leq \overline b_s(K).$$
Now, let's compare Betti numbers of $M$. We obtain
\begin{align*}
b_s(M)-b_{s-1}(M) & = \overline b_s(K) + \overline b_{s-1}(K) - \overline b_{s-1}(K)
- \overline b_{s-2}(K) \\
& = \overline b_s(K) - \overline b_{s-2}(K) \\
& \geq 0,
\end{align*}
by the argument above. Hence, the Betti numbers of $M$ increase up to the
middle dimension.
\end{proof}

In \cite{CDM} it was shown that the first Betti number of a co-K\"ahler
manifold is always odd. (Indeed, it was shown later that, for $M$ co-K\"ahler,
$S^1\times M$ is K\"ahler, so this also follows by Hard Lefschetz). Here, we can
infer this as a simple consequence of our splitting. Now, $K$ is a K\"ahler manifold,
so ${\rm dim}(H^1(K;\RR))$ is even and there is
a non-degenerate skew symmetric bilinear (i.e. symplectic) form $b\colon H^1(K;\RR) \otimes H^1(K;\RR)
\to H^{2n}(K;\RR) \cong \RR$ defined by
$$b(\alpha,\beta) = \alpha \cdot \beta \cdot \omega^{n-1}.$$
Let $G = \ZZ_m =\langle\varphi\,|\,\varphi^m=1\rangle$, note that invariance of $\omega$
implies $\varphi^*\omega=\omega$ and compute:
\begin{align*}
\varphi^*(b)(\alpha,\beta) & = b(\varphi^*\alpha,\varphi^*\beta) \\
& = \varphi^*\alpha \cdot \varphi^*\beta \cdot \omega^{n-1} \\
& = \varphi^*\alpha \cdot \varphi^*\beta \cdot \varphi^*\omega^{n-1} \\
& = \varphi^*(\alpha \cdot \beta \cdot \omega^{n-1}) \\
& = \alpha \cdot \beta \cdot \omega^{n-1} \\
& = b(\alpha,\beta),
\end{align*}
where the second last line comes from the fact that
$\alpha \cdot \beta \cdot \omega^{n-1}=k\cdot \omega^n$ and
$\varphi^*\omega^n = \omega^n$. Hence, $\varphi^*$ is a symplectic linear
transformation on the symplectic vector space $H^1(K;\RR)$. But now
the Symplectic Eigenvalue Theorem says that the eigenvalue $+1$ occurs
with even multiplicity. Thus $\overline b_1(K) = {\rm dim}(H^1(K;\RR)^G)$
is even. Hence, by \thmref{thm:betti} (1), we have the following result.

\begin{cor}\label{cor:firsthom}
The first Betti number of a compact co-K\"ahler manifold is odd.
\end{cor}

\section{Fundamental Groups of Co-K\"ahler Manifolds}\label{sec:fundgps}
An important question about compact K\"ahler manifolds is exactly what groups
arise as their fundamental groups. For instance, every finite group is
the fundamental group of a K\"ahler manifold, while a free group on more than
one generator cannot be the fundamental group of a K\"ahler manifold (see \cite{ABCKT}
for more properties of these groups).
Li's mapping torus result shows that the fundamental group of a compact co-K\"ahler
manifold is always a semidirect product of the form $H\rtimes_{\psi}\ZZ$,
where $H$ is the fundamental group of a K\"ahler manifold. As an alternative,
because the finite cover of \thmref{thm:cosympsplit} corresponds to the subgroup
${\mathrm{Ker}}(\pi_1(M) \to \ZZ_m)$, \thmref{thm:cosympsplit} implies the following.

\begin{thm}\label{thm:fundgp}
If $(M^{2n+1},J,\xi,\eta,g)$ is a compact co-K\"ahler manifold with integral structure
and splitting $M \cong K \times_{\mathbb Z_m} S^1$, then $\pi_1(M)$
has a subgroup of the form $H \times \ZZ$, where $H$ is the fundamental
group of a compact K\"ahler manifold, such that the quotient
$$\frac{\pi_1(M)}{H \times \ZZ} $$
is a finite cyclic group.
\end{thm}

\subsection{Co-K\"ahler manifolds with transversally positive definite Ricci tensor}\label{sec:b1=1}

Now let's see how to use our general approach to recover a result of De Le\'on and Marrero 
(\cite{DM}) concerning compact co-K\"ahler manifolds with transversally positive definite 
Ricci tensor. Let $(M^{2n+1},J,\xi,\eta,g)$ be an almost contact metric manifold and let 
$\F$ be the codimension 1 foliation $\ker(\eta)$. Let $T\F$ be the vector subbundle of the 
tangent bundle of $M$ consisting on vectors that are tangent to $\F$: at a point $x\in M$, then
$$
T_x\F=\{v\in T_xM \ | \ \eta_x(v)=0\}.
$$
Let $S$ be the Ricci curvature tensor of $M$. $S$ is called \emph{transversally positive 
definite} if $S_x$ is positive definite on $T_x\F$ for all $x\in M$. In \cite{DM}, the 
authors prove the following result.

\begin{thm}[{\cite[Theorem 3.2]{DM}}]\label{thm:DM}
If $M$ is a compact co-K\"ahler manifold with transversally positive definite Ricci tensor,
then $\pi_1(M)$ is isomorphic to $\ZZ$.
\end{thm}
Their result relies, in turn, on the following theorem of Kobayashi (\cite{K}).
\begin{thm}[{\cite[Theorem A]{K}}]\label{thm:Kob}
A compact K\"ahler manifold with positive definite Ricci tensor is simply connected.
\end{thm}
We now give an alternative proof of \thmref{thm:DM} from our viewpoint.

\begin{proof}
Let $(M^{2n+1},J,\xi,\eta,g)$ be a co-K\"ahler manifold and let $\F$ be the foliation 
given by $\ker(\eta)$. Assume that the Ricci curvature tensor is transversally positive.
Using Li's approach, we can pass to an integer co-K\"ahler structure and this 
process uses the flow of the Reeb vector field $\xi$ to deform the leaves of $\F$  
into the K\"ahler submanifold $K$. Now recall that $\xi$ is Killing on a co-K\"ahler 
manifold, so its flow consists of isometries of $M$. In particular, if $S$ is
transversally positive definite on $\F$, then $K$ is a K\"ahler manifold with positive 
definite Ricci tensor. By \thmref{thm:Kob}, $K$ is simply connected. Therefore $\pi_1(M)$
is the semi-direct product of the trivial group with $\ZZ$, hence isomorphic to $\ZZ$.
\end{proof}

\subsection{Co-K\"ahler manifolds with solvable fundamental group}\label{sec:solvgp}
There has been much work done in the past $20$ years regarding the question of
whether K\"ahler solvmanifolds are tori. In \cite{Has}, for instance, it is shown
that such a manifold is a finite quotient of a complex torus which is also the total
space of a complex torus bundle over a complex torus. In \cite{FV}, Hasegawa's
result was applied to show the following.

\begin{thm}
A solvmanifold has a co-K\"ahler structure if and only if it is a finite quotient
of torus which has a structure of a torus bundle over a complex torus. As a
consequence, a solvmanifold $M = G/\Gamma$ of completely solvable type has a
co-K\"ahler structure if and only if it is a torus.
\end{thm}

Note that we have changed the terminology of \cite{FV} to match ours. We can
use \thmref{thm:cosympsplit} to contribute something in this vein.

\begin{thm}\label{thm:kahlersolv}
Let $(M^{2n+1},J,\xi,\eta,g)$ be an aspherical co-K\"ahler manifold with integral
structure and suppose $\pi_1(M)$ is a solvable group. Then $M$ is a finite quotient
of a torus.
\end{thm}

\begin{proof}
We know that every aspherical solvable K\"ahler group contains a finitely generated
abelian subgroup of finite index (see \cite[section 1.5]{BC} for instance). Now,
if $M=K_\varphi$ is the Li mapping torus description of $M$, we see that $K$ is
K\"ahler and aspherical with solvable fundamental group (as a subgroup of
$\pi_1(M)$). Hence, $K$ is finitely covered by a torus. By \thmref{thm:cosympsplit},
there is a finite ${\mathbb Z_m}$-cover $K \times S^1 \to M$ and this then
displays $M$ itself as a finite quotient of a torus.
\end{proof}

\section{Automorphisms of K\"ahler manifolds}\label{sec:maptor}
In this section, we connect our results above with certain facts about
compact K\"ahler manifolds and their automorphisms. In order to do this, we first need
some general results about mapping tori.
Let $M$ be a smooth manifold and let $\varphi\colon M\to M$ be a diffeomorphism.
Let $M_{\varphi}$ denote the mapping torus of $\varphi$. We have the following result.

\begin{prop}\label{prop:mappingtorustrivial}
The mapping torus $M_{\varphi}$ is trivial as a bundle over $S^1$
(i.e. $M_{\varphi}\cong M\times S^1$ over $S^1$) if and only if
$\varphi\in\mathrm{Diff}_0(M)$, where $\mathrm{Diff}_0(M)$ denotes
the connected component of the identity of the group $\mathrm{Diff}(M)$.
\end{prop}

\begin{proof}
First assume the mapping torus is trivial over $S^1$. We have the following
commutative diagram with top row a diffeomorphism.

$$\xymatrix{
M_{\varphi} \ar[dr]_-p \ar[rr]^-f & & M \times S^1 \ar[dl]^-{\mathrm{pr}_2} \\
& S^1 &           }
$$
where $\mathrm{pr}_2(f([x,t]))=[t]=p([x,t])$. This means that $f$ maps level-wise, so we
have $f([x,t])=(g_t(x),t)$, where each $g_t\colon M \to M$ is a diffeomorphism.
The mapping torus relation $(k,0) \sim (\varphi(k),1)$ gives
$$(g_0(x),[0])=f([x,0])=f(\varphi(x),1)=(g_1(\varphi(x)),[1])=(g_1(\varphi(x)),[0]),$$
and then we have $g_0(x)=g_1(\varphi(x))$.

Define an isotopy $F\colon M \times I \to M$ by $F(x,t)=g_0^{-1} g_t(\varphi(x))$.
Then $F(x,0)=g_0^{-1} g_0(\varphi(x))=\varphi(x)$ and $F(x,1)=
g_0^{-1} g_1(\varphi(x))=g_0^{-1} g_0(x)=x$.
Hence, $\varphi$ is isotopic to the identity.

Conversely, suppose that $\varphi\in \mathrm{Diff}_0(M)$. Then there exists
a smooth map $H\colon M\times [0,1]\to M$ such that
$$
H(m,0)=m \qquad \textrm{and} \qquad H(m,1)=\varphi(m)
$$
and $H(\cdot,t)$ is a diffeomorphism for all $t\in[0,1]$; in particular, for all
$t\in[0,1]$, there exists a diffeomorphism $H^{-1}(\cdot,t)$. Define a map
$f\colon M\times S^1\to M_{\varphi}$ by
$$
f(m,[t])=[H(m,t),t];
$$
where we identify $M\times S^1=\frac{M\times[0,1]}{(m,0)\sim (m,1)}$.
It is enough to check that $f$ is well defined, as it is clearly smooth, but
this is guaranteed by our definition of $H$. Next we define an inverse
$g\colon M_{\varphi}\to M\times S^1$ by setting
$$
g([m,t])=(H^{-1}(m,t),[t]).
$$
Again, $g$ is smooth, and we must prove that it is well defined. Indeed, we have
$$
g([m,0])=(H^{-1}(m,0),[0])=(m,[0])
$$
and
$$
g([\varphi(m),1])=(H^{-1}(\varphi(m),1),[1])=(\varphi^{-1}(\varphi(m)),[1])=(m,[1]).
$$
But $[m,[0]]=[m,[1]]$ in $M \times S^1$, so $g$ is well-defined and is an inverse for $f$.
\end{proof}

\begin{rem}
For reference, we make the simple observation that, for a diffeomorphism
$\varphi\in\mathrm{Diff}_0(M)$, which is isotopic to the identity, the induced map
on cohomology $\varphi^*\colon H^*(M;\ZZ)\to H^*(M;\ZZ)$ is the identity map.
\end{rem}

The proposition suggests that, in order to obtain non-trivial examples of mapping tori,
one should consider diffeomorphisms that do not belong to the identity component of
the group of diffeomorphisms. It is then interesting to look at the groups
$\mathrm{Diff}(M)/\mathrm{Diff}_0(M)$ or $\mathrm{Diff}_+(M)/\mathrm{Diff}_0(M)$,
the latter in case one is interested in orientation-preserving diffeomorphisms.

\begin{rem}
In case $M$ is a compact complex manifold, one can replace $\mathrm{Diff}(M)$ by
the group $\mathrm{Aut}(M)$ of holomorphic diffeomorphisms of $M$. Further,
when $M$ is compact K\"ahler, one may consider the subgroup $\mathrm{Aut}_{\omega}(M)$
of elements which preserve the K\"ahler class (but not necessarily the K\"ahler form). 
In each case, the corresponding mapping torus 
is trivial if and only if the automorphism belongs to the identity component.
\end{rem}

Now let's consider the structure group of a mapping torus.
Let $M$ be a smooth manifold and let $\varphi\colon M\to M$ be a diffeomorphism.
Then the mapping torus $M_{\varphi}$ is a fibre bundle over $S^1$ with fibre $M$.
In general, the structure group of a fibre bundle $F\to E\to B$ is a subgroup $G$
of the homeomorphism group of $F$ such that the transition functions of the bundle
take values in $G$.

\begin{prop}\label{mapping:torus}
The structure group $G$ of a mapping torus $M_{\varphi}$ is the cyclic group $\langle \varphi\rangle\subset\mathrm{Diff}(M)$.
\end{prop}

\begin{proof}[Sketch of Proof {\rm (see {\cite[Section 18]{St}})}]
The mapping torus $M_{\varphi}$ is a fibre bundle over $S^1$ with fiber the
manifold $M$. We can cover $S^1$ by two open sets $U,V$ such that $U\cap V=\{U_0,U_1\}$
consists of two disjoint open sets. Then $M_{\varphi}\big|_U=M\times U$ and
$M_{\varphi}\big|_V=M\times V$, and the mapping torus is trivial over $U$ and $V$.
To describe $M_{\varphi}$ it is sufficient to give the transition function
$g\colon U\cap V\to\mathrm{Diff}(M)$. We can assume that $g$ is the identity on
$U_0$ and $g=\varphi$ on $U_1$. Then $\varphi$ generates $G$.

\end{proof}

\begin{rem}
Another way to describe the mapping torus of a diffeomorphism
$\varphi\colon M\to M$ is as the quotient of $M\times\RR$ by the group
$\ZZ$ acting on $M\times\RR$ by
$$
(m,(p,t))\mapsto (\varphi^m(p),t-m).
$$
It is then clear that the structure group of $M_{\varphi}$ is isomorphic to the
group generated by $\varphi$.
\end{rem}

Let $(K,h,\omega)$ be a compact K\"ahler manifold, where $h$ denotes the Hermitian
metric and $\omega$ is the K\"ahler form. A \emph{Hermitian isometry} is a holomorphic 
map $\varphi\colon K\to K$ such that $\varphi^*h=h$, where $h$ is the Hermitian metric of 
$K$. Note that $\varphi$ preserves both the Riemannian metric and the symplectic form 
associated to $h$. Let $\mathrm{Isom}(K,h)\subseteq\mathrm{Aut}(K)$
denote the group of Hermitian isometries of $K$ and let $\psi\in\mathrm{Isom}(K,h)$.
Then $\psi$ is a holomorphic diffeomorphism of $K$ which preserves the
Hermitian metric $h$. In particular, $\psi^*\omega=\omega$. Li's theorem \cite{Li}
says that the mapping torus of $\psi$, denoted by $K_{\psi}$ is a compact co-K\"ahler 
manifold and, conversely, compact co-K\"ahler manifolds are always such mapping tori.
We say that a mapping torus is a \emph{K\"ahler mapping torus} if it is a mapping torus
$K_\varphi$ of a Hermitian isometry $\varphi\colon K\to K$ of a K\"ahler manifold $K$.
If $K_{\psi}$ is non-trivial, then according to \propref{prop:mappingtorustrivial},
$\psi$ defines a non-zero element in
$$
H:=\mathrm{Isom}(K,h)/\mathrm{Isom}_0(K,h).
$$
Our results prove that, up to a finite covering, $K_{\psi}\cong K\times_{\ZZ_m} S^1$
(\thmref{thm:cosympsplit}), and the $\ZZ_m$ action is by translations on the $S^1$
factor. Furthermore, we get a fibre bundle $K_{\psi}\to S^1$ with structure group
the finite group $\ZZ_m$. Notice that when we display $K_{\psi}$ as a fibre bundle
with fibre $K$, the structure group of this bundle is $\langle\psi\rangle$,
the cyclic group generated by $\psi$ in $H$. We then have the following theorem.

\begin{thm}
If $K$ is a K\"ahler manifold, then all elements of the group $H$ have finite order.
\end{thm}
\begin{proof}
Pick an element $\psi\in H$ and form the mapping torus $K_{\psi}$. The discussion above
proves that $\psi$ has finite order in $H$. Since $\psi$ is arbitrary, the result follows.
\end{proof}

Indeed, Lieberman \cite{Lie} proves a much more general result, but in a much harder way.

\begin{thm}[{\cite[Proposition 2.2]{Lie}}]
Let $K$ be a K\"ahler manifold and let $\mathrm{Aut}_{\omega}(K)$ denote the group
of automorphisms of $K$ preserving a K\"ahler class (but not necessarily the K\"ahler form).
Let $\mathrm{Aut}_0(K)$ be the identity component. Then the quotient
$$\mathrm{Aut}_{\omega}(K)/\mathrm{Aut}_0(K)$$
is a finite group.
\end{thm}

\begin{rem}
In \cite{Li}, Li also shows that the almost cosymplectic manifolds of \cite{CDM}
arise as symplectic mapping tori. That is, if $M$ is almost cosymplectic in the
terminology of \cite{CDM}, then there is a symplectic manifold $S$ and a symplectomorphism
$\varphi\colon S \to S$ such that $M \cong S_\varphi$. Li calls these manifolds
\emph{co-symplectic}. By the discussion in \secref{sec:crsplit} and the results above,
we see that there is a version of \thmref{thm:cosympsplit} for Li's co-symplectic manifolds
\emph{when the defining symplectomorphism $\varphi$ is of finite order in}
$${\rm Symp}(S)/{\rm Symp}_0(S).$$
Thus, knowledge about when this can happen would be very interesting.
\end{rem}

In general, one can not expect a non-zero element in ${\rm Symp}(S)/{\rm Symp}_0(S)$ to 
have finite order. As an example, consider the torus $T^2$ with the standard symplectic 
structure and let $\varphi\colon T^2\to T^2$ be the diffeomorphism covered by the linear 
transformation $A\colon \RR^2\to\RR^2$ with matrix
$$
A= \begin{pmatrix}
2 & 1 \\
1 & 1
\end{pmatrix}
$$
Then $\varphi$ is an area-preserving diffeomorphism of $T^2$, hence a symplectomorphism. 
Notice that the action of $\varphi$ on $H^1(T^2;\RR)$, which is represented by the matrix $A$, 
is nontrivial. Hence the symplectic mapping torus $T^2_\varphi$ is not diffeomorphic to 
$T^3=T^2\times S^1$; according to \propref{prop:mappingtorustrivial}, $\varphi$ is non-zero 
in ${\rm Symp}(T^2)/{\rm Symp}_0(T^2)$. Clearly $\varphi$ has infinite order.

\section{Examples}\label{sec:cdmexample}

The first example of a compact co-K\"ahler manifold that is not homeomorphic to the 
global product of a K\"ahler manifold and $S^1$ was given in \cite{CDM}.
This example was generalized to every odd dimension in \cite{MP}. Each of these examples 
is a \emph{solvmanifold} (i.e. a compact quotient of a solvable Lie group by a lattice) and
can be described as a mapping torus of a suitable Hermitian isometry of the torus $T^{2n}$. 
Although the examples were constructed in every dimension $2n+1$, it was
not clear whether they could be the product of some compact K\"ahler manifold of 
dimension $2n$ and a circle. Of course, from what we have said above, they \emph{are} 
products \emph{up to a finite cover}.

In this section we analyze these examples from both Li's mapping torus and our finite cover
splitting points of view. We also show that these examples are never the global product 
of a compact K\"ahler manifold and a circle, thus producing,
in every odd dimension, examples of compact co-K\"ahler manifolds that are not products.

Let us begin with the CDM example. Consider the  matrix
$$A= \begin{pmatrix}
0 & 1 \\
-1 & 0
\end{pmatrix}
$$
in $GL(\ZZ,2)$ and note that it defines a K\"ahler isometry of $T^2$ which
we can write as $A(x,y)=(y,-x)$. Li's approach says to form the mapping
torus
$$T^2_A =\frac{T^2 \times [0,1]}{(x,y,0)\sim (A(x,y),1)},$$
and then $T^2_A$ is a co-K\"ahler manifold with associated fibre bundle
$T^2 \to T^2_A \to S^1$ given by the projection
$$[x,y,t] \mapsto [t].$$
Now, $A$ has order $4$, so the picture is quite simple: namely, a central circle
winds around the mapping torus $4$ times before closing up. Therefore,
we see that we have a circle action on $T^2_A$ given by
$$ S^1 \times T^2_A \to T^2_A, \qquad ([s],[x,y,t])\mapsto [x,y,t+4s].$$
When the orbit map $S^1 \to T^2_A$, $[s] \mapsto [x_0,y_0,0]$ is composed
with the projection map $T^2_A \to S^1$, we get
$$S^1 \to S^1, \qquad [s] \mapsto [4s]$$
which induces multiplication by $4$ on $H_1(S^1;\ZZ)$. Hence, the $S^1$-action
is homologically injective and \thmref{thm:cosympsplit} then gives a
finite cover of $T^2_A$ of the form $T^2 \times S^1$.
Hence, $T^2_A$ is finitely covered by a torus. Now let's
look at the Betti numbers of $T^2_A$ using \thmref{thm:betti}.

The diffeomorphism $A$ acts on $H^1(T^2;\RR)$ by the matrix $P_*=A^t$, $P_*(x,y)=(-y,x)$, 
and on $H^2(T^2;\RR)$ by the identity; hence the K\"ahler
class is invariant (as we know in general). Otherwise, there are no
invariant classes in degrees greater than zero. To see this, suppose
$P_*(ax+by)=-ay+bx=ax+by$. Thus, $a=b$ and $a=-b$, so $a=b=0$. Now we have the
following.
\vskip5pt
\begin{itemize}
\item $b_1(T^2_A) = \overline b_1(T^2) +1 = 0+1 = 1$;
\vskip5pt
\item  $b_2(T^2_A) = \overline b_2(T^2) + \overline b_1(T^2) = 1+0 = 1$;
\vskip5pt
\item  $b_3(T^2_A) = \overline b_3(T^2) + \overline b_2(T^2) = 0+1 = 1$.
\end{itemize}
\vskip5pt
As noted in \cite{CDM}, this shows that $T^2_A$ is not a global product. For,
as an orientable $3$-manifold with first Betti number $1$, there is no
other choice but $S^1 \times S^2$ and this is ruled out since the fibre
bundle $T^2 \to T^2_A \to S^1$ shows that $T^2_A$ is aspherical.\\

The CDM example also fits in the scope of \thmref{thm:fundgp}. To see this,
we compute the fundamental group of $T^2_A$ explicitly. The fibre bundle
$T^2 \to T^2_A \to S^1$ shows that we have a short exact sequence of groups
$$
0\to \ZZ^2 \to \Gamma \to \ZZ\to 0,
$$
where $\Gamma=\pi_1(T^2_A)$. Since $\ZZ$ is free, $\Gamma$ is a semidirect
product $\ZZ^2\rtimes_{\phi}\ZZ$. The action of $\ZZ$ on $\ZZ^2$ is given by
the group homomorphism $\phi\colon \ZZ\to\mathrm{SL}(2,\ZZ)$ sending $1\in\ZZ$
to $\phi(1)=A\in\mathrm{SL}(2,\ZZ)$. As we remarked above, $T^2_A$ is covered
$4:1$ by a torus $T^3$ and this covering gives a map $\psi\colon \ZZ^3\to\Gamma$.
The map $\psi$ sends $(m,n,p)\in\ZZ^3$ to $(m,n,4p)\in\Gamma$, hence the quotient
$\Gamma/\ZZ^3$ is isomorphic to $\ZZ_4$.\\

For any $n\geq 1$ we give an example of a compact co-K\"{a}hler manifold of dimension 
$(2n+1)$ which is not homeomorphic to the global product of a compact
manifold of dimension $2n$ and a circle. This example was constructed by Marrero and 
Padr\'on (see \cite{MP}, example B1). We describe it according to our mapping torus and splitting
approach.

Let $\zeta=e^{2\pi i/6}$ and consider the lattice $\Lambda\subset\CC$ spanned by $1$ 
and $\zeta$. Set $T^2=\CC/\Lambda$ and $T^{2n}=\underbrace{T^2\times\ldots\times T^2}_{n \ \textrm{times}}$.
Then $T^{2n}$ is a compact K\"ahler manifold, with K\"ahler structure inherited by $\CC^n$.
Let $B\colon T^{2n}\to T^{2n}$ be the map covered by the linear transformation 
$\tilde{B}\colon \CC^n\to\CC^n$, $\tilde{B}=\textrm{diag}(\zeta,\ldots,\zeta)$. 
Then $B$ is a Hermitian isometry of the torus $T^{2n}$. Let $T^{2n}_B$ be the mapping 
torus of the Hermitian isometry $B$. Then $T^{2n}_B$ is a co-K\"ahler manifold and the 
associated fibre bundle $T^{2n}\to T^{2n}_B\to S^1$ is given by the projection 
$[p,t]\mapsto t$, where $p\in T^{2n}$. The Hermitian isometry $B$
has order 6, so we obtain a circle action on $T^{2n}_B$ given by
$$
S^1\times T^{2n}_B\to T^{2n}_B, \quad ([s],[p,t])\mapsto ([p,t+6s]).
$$
Composing the orbit map with the projection $T^{2n}_B\to S^1$, we obtain $S^1\to S^1$, 
$s\mapsto 6s$, which induces multiplication by 6 in cohomology. The $S^1-$action is 
homologically injective, and \thmref{thm:cosympsplit} gives us a finite cover of 
$T^{2n}_B$ of the form $T^{2n}\times S^1$. Hence $T^{2n}_B$ is finitely covered by a torus.

The fundamental group of $T^{2n}_B$ is the semidirect product 
$\Gamma=\Lambda^{n}\rtimes_{\phi}\ZZ$, where the action of $\ZZ$ on $\Lambda^n$ is given 
by the group homomorphism $\phi:\ZZ\to\mathrm{SL}(\Lambda^n)$, $\phi(1)=B$. The 
commutator subgroup $[\Gamma,\Gamma]$ is $\Lambda^n$, so in particular, $\Gamma$ is a 
solvable group. The first homology of $T^{2n}_B$ is
$$
H_1(T^{2n}_B;\ZZ)\cong\frac{\Gamma}{[\Gamma,\Gamma]}\cong\ZZ,
$$
so $b_1(T^{2n}_B)=1$. Now assume that $T^{2n}_B$ is the product of a compact manifold 
$K$ and a circle, $T^{2n}_B\cong K\times S^1$. The fundamental group of $K$ is solvable, being
a subgroup of $\Gamma$. Applying the K\"unneth formula with integer coefficients to 
$T^{2n}_B=K\times S^1$, we see that $H_1(K;\ZZ)=0$. Hence, $\pi_1(K)\cong[\pi_1(K),\pi_1(K)]$, 
but this is not possible because $\pi_1(K)$ is solvable. We conclude that $T^{2n}_B$ is not 
homeomorphic to the product of a compact, $2n-$dimensional manifold and a circle.

\vskip.2in

\noindent {\bf Acknowledgments.}
We thank Greg Lupton for useful conversations and the referee for several 
helpful suggestions. The first author thanks the Department of Mathematics 
at Cleveland State University for its hospitality during his extended visit (funded by 
CSIC and ICMAT) to Cleveland. In addition, the first author was partially supported by 
Project MICINN (Spain) MTM2010-17389. The second author was partially supported by a grant 
from the Simons Foundation: (\#244393 to John Oprea).




\end{document}